\documentclass[12pt, a4paper, reqno]{amsart}

\usepackage[english]{babel}

\usepackage[T1]{fontenc}
\usepackage{lmodern}

\usepackage{amsmath}
\usepackage{amssymb}
\usepackage{amsthm}

\usepackage{url}

\usepackage{hyperref} 

\usepackage{enumerate}
\usepackage[normalem]{ulem}

\usepackage{float}
\restylefloat{table}

\renewcommand{\epsilon}{\varepsilon}
\renewcommand{\theta}[0]{\vartheta}
\renewcommand{\phi}[0]{\varphi}

\newcommand{\gammabar}[0]{\bar{\gamma}}

\newcommand{\Z}{\mathbb{Z}}

\newcommand{\Hc}[0]{\mathcal{H}}

\newcommand{\Size}[1]{\left\lvert #1 \right\rvert}

\newcommand{\Span}[1]{\left\langle\, #1 \,\right\rangle}

\newcommand{\Set}[1]{\left\{ #1 \right\}}

\newcommand{\norm}[0]{\trianglelefteq}

\DeclareMathOperator{\End}{End}
\DeclareMathOperator{\Aut}{Aut}

\DeclareMathOperator{\Frat}{Frat}
\DeclareMathOperator{\Hol}{Hol}

\DeclareMathOperator{\inv}{inv}

\newtheorem{dummy}{Dummy}
\numberwithin{dummy}{section}
\numberwithin{figure}{section}

\newtheorem{theorem}[dummy]{Theorem}

\newtheorem{remark}[dummy]{Remark}

\newtheorem{lemma}[dummy]{Lemma}

\newtheorem{proposition}[dummy]{Proposition}
\newtheorem{prop}[dummy]{Proposition}

\newtheorem{corollary}[dummy]{Corollary} 

\theoremstyle{definition}
\newtheorem{definition}[dummy]{Definition}

\newtheorem{example}[dummy]{Example}

\theoremstyle{remark}

\numberwithin{equation}{section}

\begin{document}

\date{17 July 2020, 10:57 CEST --- Version 2.08%
}

\title[Bi-Skew Braces]
      {Bi-Skew Braces and\\
        Regular Subgroups of the holomorph}
      
\author{A. Caranti}

\address[A.~Caranti]%
 {Dipartimento di Matematica\\
  Universit\`a degli Studi di Trento\\
  via Sommarive 14\\
  I-38123 Trento\\
  Italy} 

\email{andrea.caranti@unitn.it} 

\urladdr{http://science.unitn.it/$\sim$caranti/}

\subjclass[2010]{20B35 16W30}

\keywords{holomorph, multiple holomorph, skew braces, bi-skew braces,
  regular subgroups, automorphisms}  

\begin{abstract}
  L.~Childs has defined a skew brace $(G, \cdot, \circ)$ to be a bi-skew
  brace if $(G, \circ, \cdot)$ is also a skew brace, and has given
  applications of this concept to the equivalent theory of Hopf-Galois
  structures. 

  The goal of this paper is to deal with bi-skew braces
  $(G, \cdot, \circ)$ from the
  yet equivalent point of view of regular subgroups of the holomorph
  of $(G, \cdot)$. In particular, we find that certain groups studied by
  T.~Kohl, F.~Dalla Volta and the author, and C.~Tsang all yield
  examples of bi-skew braces.

  Building on a construction of Childs, we also give various methods for
  exhibiting further examples of bi-skew braces.
\end{abstract}

\thanks{The author is a member of INdAM---GNSAGA. The author
  gratefully acknowledges support from the Department of Mathematics of
  the University of Trento.}

\maketitle

\thispagestyle{empty}

\section{Introduction}
\label{sec:intro}

Let $G = (G, \cdot)$ be a group. A (right) \emph{skew brace} with
\emph{additive group} $(G, \cdot)$
is a triple $(G, \cdot, \circ)$, where $\circ$ is an operation on $G$
such that $(G, \circ)$ is also a group, and the following axiom holds
\begin{equation}
  \label{eq:axiom}
  (x y) \circ z
  =
  (x \circ z) \cdot z^{-1} \cdot (y \circ z),
  \qquad
  \text{for $x, y, z \in G$.}
\end{equation}
The group $(G,  \circ)$ is called the \emph{circle group}  of the skew
brace.

According  to L.~Childs~\cite{Childs-bi-skew},  such a  skew brace  is
called a  \emph{bi-skew brace} if $(G,  \circ, \cdot)$ is also  a skew
brace.  Given an arbitrary group  $(G, \cdot)$, the trivial skew brace
$(G, \cdot, \cdot)$ is clearly a bi-skew brace. If $G$ is non-abelian,
setting $x  \circ y = y  \cdot x$ we obtain  another (trivial) example
$(G, \cdot, \circ)$ of a bi-skew brace.

It is well known that the specification of a skew brace  with
additive group $(G,  \cdot)$ corresponds to the choice  of a regular
subgroup of  the holomorph  $\Hol(G)$ of $(G,  \cdot)$; we  recall the
details of this correspondence in Section~\ref{sec:preliminaries}. The
two trivial examples of bi-skew braces we have  mentioned in the
previous paragraph
correspond to the images of the right resp.\ left regular
representation of $(G, \cdot)$.

The goal of  this paper is to  discuss the concept of  a bi-skew brace
from the  point of view  of regular  subgroups, and of  the associated
gamma functions, as  defined and used by  F.~Dalla Volta, E.~Campedel,
I.~Del  Corso  and  the  author in~\cite{affine,  fgab,  perfect,  p4,
  CCDC}.    From   this   discussion,   which   we   carry   out   in
Section~\ref{sec:main}, it will follow  among others that the multiple
holomorphs  studied by  T.~Kohl~\cite{Kohl-multi}, F.~Dalla  Volta and
the   author~\cite{fgab,    perfect},   the    author~\cite{p4},   and
C.~Tsang~\cite{Tsang-p} all yield examples of bi-skew braces.

Finally,    building    on    a    construction    of    Childs,    in
Section~\ref{sec:constructions}   we  exhibit   various  methods   for
constructing  more examples  of bi-skew  braces, and  review from  the
point  of  view  of  bi-skew braces  results  of  J.~C.~Ault~and
J.~F.~Watters~\cite{AuWa} about the occurrence  of nilpotent groups of
nilpotence class $2$ as circle groups of skew braces.

\section*{Acknowledgement}

We are  indebted to the  referee for a number  of insightful
comments,   which   have   greatly  contributed   to   improving   the
exposition. In particular, Examples~\ref{ex:ex}, \ref{ex:exexex}, and
\ref{ex:exex} stem 
from suggestions of the referee.

\section{Preliminaries}
\label{sec:preliminaries}

We reprise the following from~\cite{CCDC}.
Let $G = (G, \cdot)$ be a group. There is a one-to-one correspondence
between binary operations $\circ$ on $G$, and maps $\gamma : G \to
G^{G}$, where $G^{G}$ is the set of maps from $G$ to $G$, given by
\begin{equation*}
  x^{\gamma(y)} = (x \circ y) \cdot y^{-1},
  \quad\text{and}\quad
  x \circ y = x^{\gamma(y)} y,
  \qquad\text{for $x, y \in G$.}
\end{equation*}
In Table~\ref{table:1}  we exhibit a correspondence between
certain properties of the operation $\circ$, and certain properties of
the corresponding map $\gamma$.
\begin{table}[H]
  \caption{$\circ$ and $\gamma$}
  \label{table:1}
  \begin{center}
    \begin{tabular}{p{4.50cm}|p{7.00cm}}
      \multicolumn{1}{c|}{Property of $\circ$}
      &
      \multicolumn{1}{c}{Property of $\gamma$}
      \tabularnewline\hline\hline
      \raggedright
      Axiom~\eqref{eq:axiom} holds
      &
      \raggedright
      $\gamma(g) \in \End(G, \cdot)$, for $g \in G$
      \tabularnewline\hline
      $\circ$ is associative
      &
      \raggedright
      $\gamma(x^{\gamma(y)} y) = \gamma(x) \gamma(y)$ holds,
      for $x, y \in G$
      \tabularnewline\hline
      $\circ$ admits inverses
      &
      \raggedright
      $\gamma(g)$ is bijective, for $g \in G$
    \end{tabular}
  \end{center}
\end{table}
The properties  on the  first line are  equivalent. The
properties on  the second  line are  equivalent, under  the assumption
that the properties on the first line hold.  On
the third line, the property on  the right implies the property on the
left, while to prove the  left-to-right implication one need to assume
the  property on  the right  in the  second line.
The  fact that  $(G, \circ)$  has the  same identity  as $(G,  \cdot)$
follows from the properties in the first line.

Taken  together, the  properties in  the left  column state  that $(G,
\cdot, \circ)$ is a (right) skew brace. A function $\gamma$ satisfying
the properties in the right column is called a gamma function, or a GF
for short. We will actually need the following
\begin{definition}[\protect{\cite[Definition 2.1]{CCDC}}]
  \label{def:GF}
  Let $G$ be a group, $A \le G$, and $\gamma: A \to \Aut(G)$ a
  function.

  $\gamma$ is said to satisfy the \emph{gamma functional
    equation} (or \emph{GFE} for short) if
  \begin{equation*}
    \gamma( g^{\gamma(h)} h )
    =
    \gamma(g) \gamma(h),
    \qquad
    \text{for all $g, h \in A$.}
  \end{equation*}
  
  $\gamma$ is said to be a \emph{relative gamma function} (or
  \emph{RGF} for short) on $A$ if it satisfies the gamma functional
  equation, and $A$ is $\gamma(A)$-invariant.

  If $A = G$, a relative gamma function is simply called a
  \emph{gamma function} (or \emph{GF} for short) on $G$.
\end{definition}
In   Section~\ref{sec:main}  we   will   be   using  gamma   functions
only; relative gamma  functions will be used for  two examples in
Section~\ref{sec:constructions}.

\begin{remark}
 \label{rem:ker}
  Note that, referring to Table~\ref{table:1}, a GF $\gamma$  on $G$ is a
  homomorphism  of  groups  $\gamma  : (G,  \circ)  \to  \Aut(G, \cdot)$.  In
  particular $\ker(\gamma)  = \Set{ g  \in G : \gamma(g)  = 1 }$  is a
  normal subgroup of  $(G, \circ)$, but in general only  a subgroup of
  $(G, \cdot)$.
\end{remark}

Recall that the right regular representation is the homomorphism
\begin{align*}
  \rho :\ &G \to S(G)\\
          &g \mapsto (x \mapsto x g),
\end{align*}
where  $S(G)$  is the  group  of  permutations  on  the set  $G$.  The
(permutational) holomorph of  a group $G$ is the  normaliser in $S(G)$
of the image of the right regular representation $\rho$ of $G$,
\begin{equation*}
  \Hol(G)
  =
  N_{S(G)} (\rho(G))
  =
  \Aut(G) \rho(G),
\end{equation*}
where the latter is a semidirect product.

Let  $N$  be  a  regular  subgroup  of  $\Hol(G)$,  that  is,  $N$  is
transitive, and the one-point stabilisers  are trivial. The map $N \to
G$ which takes $n  \in N$ to the image $1^{n}$ of $1  \in G$ under $n$
is thus a bijection. Let $\nu : G \to N$ be its inverse. Thus $\nu(g)$
is the unique  element of $N$ such  that $1^{\nu(g)} = g$,  for all $g
\in  G$. Since  $N  \le \Hol(G)$,  we can  write  $\nu(g) =  \gamma(g)
\rho(g)$, for  a suitable  function $\gamma  : G  \to \Aut(G)$.  It is
immediate  to  see  that  $\gamma$  is  a  gamma  function,  and  that
conversely every gamma function $\gamma$ determines a regular subgroup
$N$  of $\Hol(G)$  as $N  = \Set{\gamma(g)  \rho(g) :  g \in  G}$.  It
follows  that  the  data  of  Table~\ref{table:1}  are  equivalent  to
specifying a regular  subgroup $N$ of the holomorph  $\Hol(G)$, as the
set of maps
\begin{equation*}
  N
  =
  \Set{ x \mapsto x \circ y : y \in G }
  =
  \Set{ \gamma(y) \rho(y) : y \in G},
\end{equation*}
as for $x, y \in G$ one has $x^{\gamma(y) \rho(y)} = x^{\gamma(y)} y =
x \circ y$.
For such an $N$, the map
\begin{align*}
  \nu :\ &(G, \circ) \to N\\
         &y \mapsto (x \mapsto x \circ y)
\end{align*}
is an isomorphism of groups.

For the details, see the discussions in~\cite{perfect, CCDC}.

\begin{remark}
  \label{rem:one-to-three}
  In the rest of the paper, we  will freely use the fact that, given a
  group $(G,  \cdot)$, any of  the following data  uniquely determines
  one of the others
  \begin{enumerate}
  \item an operation  $\circ$ on $G$ such that $(G,  \cdot, \circ)$ is a
    skew brace,
  \item a regular subgroup $N \le \Hol(G)$, and
  \item a gamma function $\gamma : G \to \Aut(G)$.
  \end{enumerate}
\end{remark}

We also recall from~\cite[Proposition 2.21]{CCDC} that, given a
regular subgroup $N \le 
\Hol(G)$ with associated GF $\gamma$, the GF associated to the regular
subgroup $N^{\inv} \le \Hol(G)$ is
\begin{equation*}
  \gammabar(y)
  =
  \gamma(y^{-1}) \rho(y^{-1}).
\end{equation*}
Here $\inv \in S(G)$ is the inverse map $x \mapsto x^{-1}$ on
$G$, which normalises $\Hol(G)$, and thus acts by conjugacy on the set
of regular subgroups of $\Hol(G)$. This construction is equivalent to
the opposite skew brace 
construction of~\cite{opposite}.

\section{Regular subgroups of the holomorph, and gamma functions}
\label{sec:main}

\subsection{Regular subgroups}

A skew brace $(G, \cdot, \circ)$ is said to be a bi-skew brace if $(G,
\circ, \cdot)$  is also  a skew brace~\cite{Childs-bi-skew}.

We reinterpret this definition in terms of regular subgroups and gamma
functions. Let
\begin{enumerate}
\item
  $\gamma$ be the GF associated to $(G, \cdot, \circ)$, that is
  \begin{equation*}
    x \circ y = x^{\gamma(y)} y,
    \text{ for $x, y \in G$,}
  \end{equation*}
  and
\item 
  $\gamma'$ the GF associated to $(G, \circ, \cdot)$, that is
  \begin{equation*}
    x y = x^{\gamma'(y)} \circ y,
    \text{ for $x, y \in G$.}
  \end{equation*}
\end{enumerate}
Therefore, for $x, y \in G$, we have
\begin{equation*}
  x y
  =
  x^{\gamma'(y)} \circ y
  =
  x^{\gamma'(y) \gamma(y)} y,
\end{equation*}
so that for $y \in G$ we have
\begin{equation}
  \label{eq:delta-is-inverse-of-gamma}
  \gamma'(y) = \gamma(y)^{-1}.
\end{equation}
This  implies,  as  per  Remark~\ref{rem:ker},  that  $\ker(\gamma)  =
\ker(\gamma')$ is normal in both $(G, \cdot)$ and $(G, \circ)$.

We now state the following theorem, in which we use the
convention of Remark~\ref{rem:one-to-three}, and write
\begin{align*}
  \iota :\ &G \to \Aut(G)\\
           &y \mapsto (x \mapsto y^{-1} x y)
\end{align*}
for the map taking $y \in G$ to the inner automorphism of $G$ it induces.

\begin{theorem}
  \label{eq:bi-skew-via-gamma}
  Let $G = (G, \cdot)$ be a group.

  The following data are equivalent: 
  \begin{enumerate}
  \item
    \label{item:bi-skew}
    A bi-skew brace $(G, \cdot, \circ)$;
  \item
    \label{item:regular}
    a  regular subgroup  $N$ of $\Hol(G)$  which is  normalised by
    $\rho(G)$;
  \item
    \label{item:anti}
    a GF $\gamma : G  \to \Aut(G)$ which satisfies, 
    \begin{equation*}
      \gamma(x y) = \gamma(y) \gamma(x),
      \text{ for $x, y  \in G$,} 
    \end{equation*}
    that is, $\gamma$ is an
    anti-homomorphism;
  \item
    \label{item:equivariant}
    a GF $\gamma : G  \to \Aut(G)$ which satisfies
    \begin{equation*}
      \gamma(x^{\gamma(y)}) = \gamma(x)^{\gamma(y)},
      \text{ for $x, y \in G$;}
    \end{equation*}
  \item
    \label{item:no-GF}
    a function $\gamma : G  \to \Aut(G)$ which satisfies,  for $x, y \in G$,
    \begin{equation*}
      \begin{cases}
        \gamma(x y) = \gamma(y) \gamma(x),
        \\
        \gamma(x^{\gamma(y)}) = \gamma(x)^{\gamma(y)};
      \end{cases}
    \end{equation*}
  \item
    \label{item:norm0}
    a GF $\gamma : G  \to \Aut(G)$ which satisfies
    \begin{equation*}
      \gamma( [G, \gammabar(G)] ) = 1;
    \end{equation*}
  \item 
    \label{item:norm1}
    a GF $\gamma : G  \to \Aut(G)$ which satisfies
    \begin{equation*}
      \gamma(x^{-1} y^{-1} x^{\gamma(y)} y) = 1,
      \text{ for $x, y  \in G$;}
    \end{equation*}
  \end{enumerate}
\end{theorem}

\begin{definition}
  We will refer to a GF on $G$ which satisfies the conditions of
  Theorem~\ref{eq:bi-skew-via-gamma} as a \emph{bi-GF}.
\end{definition}

\begin{proof}[Proof of Theorem~\ref{eq:bi-skew-via-gamma}]
  To say that $(G, \cdot, \circ)$ is a skew brace is equivalent to
  saying that $\gamma$ maps $G$ to $\Aut(G)$, and that $\gamma$ satisfies
  the GFE
  \begin{equation*}
    \gamma( x^{\gamma(y)} y )
    =
    \gamma(x) \gamma(y),
    \qquad
    \text{for $x, y \in G$.}
  \end{equation*}
  If $(G, \circ, \cdot)$ is also a skew brace, then the values of
  $\gamma'$, and by~\eqref{eq:delta-is-inverse-of-gamma} also those of $\gamma$,
  are in $\Aut(G, \circ)$. Therefore we have, for $x, y, z \in G$,
  \begin{equation*}
    (x \circ y)^{\gamma(z)}
    =
    x^{\gamma(z)} \circ y^{\gamma(z)}
    =
    x^{\gamma(z) \gamma(y^{\gamma(z)})} y^{\gamma(z)}
  \end{equation*}
  and also
  \begin{equation*}
    (x \circ y)^{\gamma(z)}
    =
    (x^{\gamma(y)} y)^{\gamma(z)}
    =
    x^{\gamma(y) \gamma(z)} y^{\gamma(z)}.
  \end{equation*}
  It thus follows that for $y, z \in G$ we have
  \begin{equation}
    \label{eq:weak}
    \gamma(y^{\gamma(z)})
    =
    \gamma(z)^{-1} \gamma(y) \gamma(z)
    =
    \gamma(y)^{\gamma(z)},
  \end{equation}
  where in the last term we have used the notation $u^{v} = v^{-1} u v$
  for the conjugate in a group.
  We have shown that~\eqref{item:bi-skew} implies~\eqref{item:equivariant}.

  \eqref{item:equivariant} implies that $\gamma : G \to \Aut(G)$ is an
  anti-homomorphism, that is, \eqref{item:anti}, as
  \begin{equation}
    \label{eq:anti}
    \gamma(x y)
    =
    \gamma(x^{\gamma(y)^{-1}}) \gamma(y)
    =
    \gamma(x)^{\gamma(y)^{-1}} \gamma(y)
    =
    \gamma(y) \gamma(x).
  \end{equation}
  Conversely, \eqref{item:anti} implies~\eqref{item:equivariant}, so that
  the two conditions are equivalent. In fact,
  assuming~\eqref{item:anti} we have
  \begin{equation*}
    \gamma(x^{\gamma(y)^{-1}})
    =
    \gamma(x y) \gamma(y)^{-1}
    =
    \gamma(y) \gamma(x) \gamma(y)^{-1}
    =
    \gamma(x)^{\gamma(y)^{-1}}.
  \end{equation*}
  (Therefore conditions~\eqref{eq:weak}~and \eqref{eq:anti} are a weaker
  form of the 
  conditions of Theorem~5.2(2) 
  of~\cite{perfect}, see Subsection~\ref{subsec:stronger} below.)

  As~\eqref{item:anti}~and \eqref{item:equivariant} are equivalent, each
  of them 
  implies~\eqref{item:no-GF}. Conversely, if~\eqref{item:no-GF} holds,
  we have, for $x, y \in G$, 
  \begin{equation*}
    \gamma(x^{\gamma(y)} y)
    =
    \gamma(y) \gamma(x^{\gamma(y)})
    =
    \gamma(y) \gamma(x)^{\gamma(y)}
    =
    \gamma(x) \gamma(y),
  \end{equation*}
  that is, $\gamma$ is a GF, so that~\eqref{item:anti}~and
  \eqref{item:equivariant} hold.

  If     $\gamma$
  satisfies~\eqref{item:equivariant},  and thus  also~\eqref{item:anti},
  then the map $\gamma' : y \mapsto (x \mapsto x^{\gamma(y)^{-1}})$ is a
  GF on $(G, \circ)$, whose associated operation is ``$\cdot$''. In fact
  we have, for $x, y, z \in G$,
  \begin{align*}
    (x \circ y)^{\gamma(z)^{-1}}
    &=
    (x^{\gamma(y)} y)^{\gamma(z)^{-1}}
    \\&=
    x^{\gamma(y) \gamma(z)^{-1}} y^{\gamma(z)^{-1}}
    \\&=
    (x^{\gamma(z)^{-1}})^{\gamma(y^{\gamma(z)^{-1}})} y^{\gamma(z)^{-1}}
    \\&=
    x^{\gamma(z)^{-1}} \circ y^{\gamma(z)^{-1}},
  \end{align*}
  so that $\gamma(z)^{-1} \in \Aut(G, \circ)$. Also,
  \begin{equation*}
    x^{\gamma(y)^{-1}} \circ y
    =
    x^{\gamma(y)^{-1} \gamma(y)}  y
    =
    x \cdot y,
  \end{equation*}
  and finally
  \begin{align*}
    \gamma( x^{\gamma(y)^{-1}} \circ y )^{-1}
    &=
    \gamma( x y )^{-1}
    \\&=
    (\gamma(y) \gamma(x))^{-1}
    \\&=
    \gamma(x)^{-1} \gamma(y)^{-1},
  \end{align*}
  that is, $\gamma'$ is a GF on $(G, \circ)$. This shows that either
  of~\eqref{item:equivariant}~or \eqref{item:anti}
  implies~\eqref{item:bi-skew}. 

  Let
  \begin{equation*}
    N
    =
    \Set{
      \gamma(x) \rho(x)
      :
      x \in G
    }
  \end{equation*}
  be the regular subgroup of $\Hol(G)$ associated to the skew brace $(G,
  \cdot, \circ)$ and the GF $\gamma$.
  For $x, y \in G$ we have
  \begin{align*}
    [\rho(x), \gamma(y) \rho(y)]
    &=
    \rho(x)^{-1} \rho(y)^{-1} \gamma(y)^{-1} \rho(x) \gamma(y) \rho(y)
    \\&=
    \rho(x^{-1} y^{-1} x^{\gamma(y)} y).
  \end{align*}
  Thus $N$ is normalised by $\rho(G)$ if and only if
  \begin{equation*}
    \gamma(x^{-1} y^{-1} x^{\gamma(y)} y) = 1.
  \end{equation*}
  Therefore~\eqref{item:regular}~and \eqref{item:norm1} are equivalent.

  \eqref{item:norm0} is a restatement of~\eqref{item:norm1}, as
  \begin{equation*}
    x^{-1} y^{-1} x^{\gamma(y)} y
    =
    x^{-1} x^{\gamma(y) \iota(y)}
    =
    [x, \gammabar(y^{-1})].
  \end{equation*}

  Assuming~\eqref{item:anti}~and \eqref{item:equivariant},
  we get~\eqref{item:norm1}, as
  for $x, y \in G$ we have
  \begin{align*}
    \gamma(x^{-1} y^{-1} x^{\gamma(y)} y)
    =
    \gamma(y) \gamma(x)^{\gamma(y)} \gamma(y)^{-1} \gamma(x)^{-1}
    =
    1
  \end{align*}

  We conclude the proof by showing that~\eqref{item:norm1}
  implies~\eqref{item:anti}. Property~\eqref{item:norm1} can be
  rewritten as
  \begin{equation*}
    \gamma( (y x)^{-1} (x^{\gamma(y)} y) ) = 1
  \end{equation*}
  for all $x, y \in G$, that is, there is $k \in \ker(\gamma)$ such
  that
  \begin{equation*}
    x^{\gamma(y)} y = y x k,
  \end{equation*}
  so that we have, for all $x, y \in G$,
  \begin{equation*}
    \gamma(x) \gamma(y) = \gamma(x^{\gamma(y)} y) = \gamma((y x) k)
    = \gamma((y x)^{\gamma(k)^{-1}}) \gamma(k) = \gamma(y x),
  \end{equation*}
  that is, \eqref{item:anti} holds.
\end{proof}
The equivalence of~\eqref{item:anti}, \eqref{item:equivariant}~and
\eqref{item:no-GF} can be stated as the following
\begin{lemma}
  Let $G$ be a group, and $\gamma : G \to \Aut(G)$ a function.

  Any two of the following conditions imply the third.
  \begin{enumerate}
  \item $\gamma$ is a GF on $G$,
  \item $\gamma$ is an anti-homomorphism, and
  \item $\gamma(x^{\gamma(y)}) = \gamma(x)^{\gamma(y)}$, for $x, y \in G$.
  \end{enumerate}
\end{lemma}

\subsection{Cubes}

In~\cite{affine, FCC} it  is proved that if $(G, +,  0)$ is an abelian
group,  then the  abelian regular  subgroups  $N \le  \Hol(G)$ can  be
described via  the structures  of commutative,  radical rings  $(G, +,
*)$.  In~\cite{affine},  the condition  that $N$  normalises $\rho(G)$
was shown to be equivalent in this context to
\begin{equation}
  \label{eq:cube}
  G * G * G = \Set{ 0 }
\end{equation}
We now show that Theorem~\ref{eq:bi-skew-via-gamma}\eqref{item:norm1}
translates as expected to~\eqref{eq:cube} when $G$ and $N$ are abelian.
Recalling
from~\cite{affine, FCC} that we have
\begin{equation*}
  x^{\gamma(y)}
  =
  x + x * y,
\end{equation*}
the additive version of
Theorem~\ref{eq:bi-skew-via-gamma}\eqref{item:norm1} reads
\begin{align*}
  0
  &=
  x^{\gamma(y^{\gamma(z)} - y)} - x
  \\&=
  x^{\gamma(y * z)} - x  
  \\&=
  x * y * z
\end{align*}
for $x, y, z \in G$, that is, \eqref{eq:cube} holds.

\subsection{A stronger condition}
\label{subsec:stronger}

In~\cite{fgab,  perfect, p4}  the  regular subgroups  $N$ of  $\Hol(G)$
which  are  normal in  $\Hol(G)  =  \Aut(G)  \rho(G)$ were  studied.
Clearly      this      is      a     stronger      condition      than
Theorem~\ref{eq:bi-skew-via-gamma}\eqref{item:regular}.  We  also note
that  regular  subgroups $N  \norm  \Hol(G)$  correspond to  the  GF's
$\gamma$ on $G$ which satisfy
\begin{equation*}
  \gamma(x^{\beta})
  =
  \gamma(x)^{\beta}
  \text{ for $x \in G$, $\beta \in \Aut(G)$,}
\end{equation*}
which is again a stronger condition than
Theorem~\ref{eq:bi-skew-via-gamma}\eqref{item:equivariant}.

We obtain
the following. 
\begin{theorem}
  \label{thm:beta}
  Let $G = (G, \cdot)$ be a group.

  The following data are equivalent: 
  \begin{enumerate}
  \item
    \label{item:auto-is-auto}
    A skew brace $(G, \cdot, \circ)$, such that $\Aut(G, \cdot)
    \le \Aut(G, \circ)$;
  \item
    \label{item:unique-iso}
    a skew brace $(G, \cdot, \circ)$ which is unique of its
    isomorphism type among skew braces with additive group $(G, \cdot)$;
  \item
    \label{item:N-normal}
    a  regular subgroup  $N \norm \Hol(G)$;
    \item
    \label{item:N-normal2}
    a  regular subgroup  $N \le \Hol(G)$ which is normalised by $\Aut(G)$;
  \item
    \label{item:beta1}
    a GF $\gamma : G \to \Aut(G)$ which satisfies
    \begin{equation*}
      \gamma(x^{\beta})
      =
      \gamma(x)^{\beta}
      \text{ for $x \in G$, $\beta \in \Aut(G)$;}
    \end{equation*}
   \item
    \label{item:beta2}
    a function $\gamma : G \to \Aut(G)$ which satisfies
    \begin{equation*}
      \begin{cases}
        \gamma(x y) = \gamma(y) \gamma(x),
        \text{ for $x, y \in G$, and}
        \\
        \gamma(x^{\beta})
        =
        \gamma(x)^{\beta}
        \text{ for $x \in G$, $\beta \in \Aut(G)$.}
      \end{cases}
    \end{equation*}
  \end{enumerate}
  The skew braces appearing in the statement are all bi-skew
  braces.
\end{theorem}

In a number  of papers, the regular subgroups which  are normal in the
holomorph of a group in certain classes of groups have been described,
thereby providing  examples of  non-trivial bi-skew  braces satisfying
the  conditions of  the Theorem.   In Table~\ref{table:2}  we list  these
papers, briefly describing the corresponding  classes of groups for the
convenience of the reader.
\begin{table}[H]
  \caption{Examples}
  \label{table:2}
  \begin{tabular}{|p{2cm}|p{10cm}|}
    \hline
    \cite{Kohl-multi} & Dihedral and quaternionic groups\\
    \hline
    \cite{fgab} & Finitely generated abelian groups\\
    \hline
    \cite{perfect} & Finite, perfect, centreless groups\\
    \hline
    \cite{p4} & Finite $p$-groups of nilpotence class two\\
    \hline
    \cite{Tsang-p} & Finite groups of squarefree orders, and finite
    $p$-groups of nilpotence class $\null < p$\\
    \hline
  \end{tabular}
\end{table}

\begin{proof}[Proof of Theorem~\ref{thm:beta}]
  \eqref{item:auto-is-auto}, \eqref{item:N-normal}, \eqref{item:beta1}~and
  \eqref{item:beta2} are shown to be 
  equivalent in~\cite{perfect}.

  Since $N$ is regular, we have
  \begin{equation}
    \label{eq:reg-norm}
    \Hol(G) =  \Aut(G) \rho(G) = \Aut(G) N,
  \end{equation}
  as   $\Aut(G)$  is   the  stabiliser   of  $1$   in  $\Hol(G)$,   so
  that~\eqref{item:N-normal} is equivalent to~\eqref{item:N-normal2}.
  
  As  to~\eqref{item:unique-iso},  the  isomorphism  classes  of  skew
  braces $(G,  \cdot, \circ)$ correspond  to the conjugacy  classes of
  regular  subgroups  $N$  of  $\Hol(G)$,  or  equivalently,  according
  to~\eqref{eq:reg-norm}, the orbits of  these regular subgroups under
  the action of $\Aut(G)$.
\end{proof}

We now recall a result of Miller~\cite{Miller-multi}.
\begin{theorem}
  The group
  \begin{equation*}
    T(G)
    =
    N_{S(G)}(\Hol(G)) / \Hol(G)
  \end{equation*}
  acts regularly on the set
  \begin{equation*}
    \Hc(G)
    =
    \Set{ N \le S(G) : \text{$N$ is regular, $N \cong G$ and $N_{S(G)}(N)
        = \Hol(G)$} }.
  \end{equation*}

  In particular, the map
  \begin{align*}
    \mathcal{T} :\ &T(G) \to \Hc(G)\\
                   &\Hol(G) \tau \mapsto \rho(G)^{\tau}.
  \end{align*}
  is well defined and bijective.
\end{theorem}
Note that $\rho(G)^{\tau} = \tau^{-1} \rho(G) \tau$ denotes the
conjugate of the subgroup $\rho(G)$ of $S(G)$ by the element $\tau \in S(G)$.

The inversion map  $\inv : x \mapsto x^{-1}$  normalises $\Hol(G)$, as
it centralises $\Aut(G)$ and  satisfies $\rho(G)^{\inv} = \lambda(G)$.
If $G$  is non-abelian, then  $\inv \notin \Aut(G)$, so  that $\Hol(G)
\inv$ is an element of $T(G)$ different from the identity.

Note  that  a regular  subgroup  $N  \in  \Hc(G)$ satisfies  $N  \norm
\Hol(G)$, and thus yields a bi-skew brace. We obtain
\begin{corollary}
  \label{cor:T}
  Let $G = (G, \cdot)$ be a group.

  \begin{enumerate}
  \item
    $\mathcal{T}$ yields an injective map from $T(G)$ to the set of
    isomorphism classes of 
    bi-skew braces of 
    Theorem~\ref{thm:beta}.
  \item 
    In  particular,  
    \begin{enumerate}
    \item
      if $G$ is non-abelian, and $\Size{T(G)} > 2$, or
    \item 
      if $G$ is abelian and $T(G) \ne \Set{1}$,
    \end{enumerate}
    then  there  are non-trivial  examples  of  bi-skew braces  with
    additive group $G$.
  \end{enumerate}
  
  However, 
  \begin{enumerate}
  \item
    \label{item:non-Phi}
    there are non-abelian groups $(G  , \cdot)$ which are the additive
    group of  a bi-skew brace, such  that the latter corresponds  to a
    regular  subgroup  $N$ which  is  not  in  the image  $\Hc(G)$  of
    $\mathcal{T}$, and
  \item
    \label{item-TG1-still}
    there are abelian  groups $(G, \cdot)$ with $T(G)  = \Set{1}$, for
    which  there are  non-trivial bi-skew  braces with  additive group
    $G$.
  \end{enumerate}
\end{corollary}
We postpone the proof of the last two statements to
Subsection~\ref{subs:proof}.

Note that the bi-skew braces $(G, \cdot, \circ)$ in the image of $\mathcal{T}$
all have $(G, \cdot) \cong (G, \circ)$. The referee has asked
whether the converse holds. In Example~\ref{ex:ex}, we show that
this is not the case.

We first note
the following Lemma, whose proof is immediate, and which will also be
used in
Subsection~\ref{subs:bi-hom}.
\begin{lemma}
  \label{prop:gammaG-is-abelian}
  Let $G$ be a group, and $\gamma$ a GF on $G$. Any two of the
  following conditions imply the third.

  \begin{enumerate}
  \item $\gamma : (G, \cdot) \to \Aut(G, \cdot)$ is a homomorphism,
  \item $\gamma(G)$ is abelian,
  \item $\gamma : (G, \cdot) \to \Aut(G, \cdot)$ is an anti-homomorphism.
  \end{enumerate}
\end{lemma}
The     third     condition     states,     according     to
Theorem~\ref{eq:bi-skew-via-gamma}, that $\gamma$ is a bi-GF.
\begin{example}
  \label{ex:ex}
  Let $p > 2$ and $q$ be primes such that $q \mid p - 1$. Let $G$
  be the non-trivial semidirect product of a
 cyclic group of order $p^{2}$ by a cyclic group of order $q$.

 In~\cite[Proposition~4.5]{CCDC} it  is shown  that there  are regular
 subgroups $N  \le \Hol(G)$ such  that $(G,  \circ) \cong N  \cong (G,
 \cdot)$,   and  $N_{\Hol(G)}(N)$ has index  $p^{2}$ in $\Hol(G)$;   in
 particular, $N \not\norm \Hol(G)$, so that $N$ is not in the image of
 $\mathcal{T}$.   We will  show that  the GF  associated to  these $N$
 satisfy Theorem~\ref{eq:bi-skew-via-gamma}\eqref{item:anti}.

 In~\cite[4.4.1]{CCDC},  the GF  $\gamma$ associated  to these  regular
 subgroups  $N$  are described.   These  are the morphisms  $\gamma : (G,
 \cdot) \to \Aut(G,  \cdot)$ which have the Sylow  $p$-subgroup of
 $G$  as  their  kernel,  and  have  thus  a  cyclic  image  of  order
 $q$.      According       to      Lemma~\ref{prop:gammaG-is-abelian},
 condition~\eqref{item:anti}   of   Theorem~\ref{eq:bi-skew-via-gamma}
 holds, so that $\gamma$ is a bi-GF.
\end{example}

\section{Some constructions}
\label{sec:constructions}

We  now  record  some  constructions  for  bi-skew  braces,  which  in
particular  cover,  and  generalise, the  examples  of~\cite{perfect,  p4,
  Tsang-p}.

\subsection{A result of Childs}

We begin with two generalisations of the following result
of Childs, which we reformulate in our terminology.
\begin{proposition}[\protect{\cite[Proposition 7.1]{Childs-bi-skew}}]
  \label{prop:Childs}
  Let the group $G$ be the semidirect product of $K
  \norm G$ by $H \le G$.

  Then the map
  \begin{align*}
    \gamma :\ &G \to \Aut(G)\\
              &h k \mapsto \iota(h^{-1})
  \end{align*}
  for $h \in H$ and $k \in K$, is a bi-GF on $G$.
\end{proposition}

We recall a result from~\cite[Proposition 2.14]{CCDC}.
\begin{proposition}
  \label{prop:lifting}
  Let $G$ be a group, and $H, K \le G$ such that $G = H K$.

  Let $\gamma': H \to \Aut(G)$ be a RGF such that
  \begin{enumerate}
  \item
    \label{item:cap}
    $\gamma'(H \cap K) \equiv 1$,
  \item
    \label{item:g-i}
    $K$ is invariant under $\{ \gamma'(h) \iota(h) : a \in H \}$.
  \end{enumerate}
  Then the map 
  \begin{equation*}
    \gamma(h k) = \gamma'(h), \text{ for $h \in H$, $k \in K$,}
  \end{equation*}
  is a well defined GF on $G$, and $\ker(\gamma) = \ker(\gamma') K$.
\end{proposition}

\subsection{Allowing central intersections}

Our   first   result   is    a   straightforward   generalisation   of
Proposition~\ref{prop:Childs},  in  which  we  allow $H$  and  $K$  to
intersect centrally.
\begin{proposition}
  \label{prop:central}
  Let $(G, \cdot)$ be a group. Let $H, K \le G$ such that
  \begin{enumerate}
  \item $K \norm G$,
  \item $H K = G$, and
  \item $H \cap K \le Z(G)$.
  \end{enumerate}
  Then the map $\gamma : G \to \Aut(G)$
  given by, for $h \in H$ and $k \in K$,
  \begin{equation}
    \label{eq:a_lifting}
    \gamma(h k) = \iota(h^{-1})
  \end{equation}
  is well defined, is a GF on $G$, and defines a bi-skew brace $(G,
  \cdot, \circ)$.

  Moreover, the following are equivalent:
  \begin{enumerate}
  \item $\gammabar$ is also a bi-GF, and
  \item $H \norm G$.
  \end{enumerate}
\end{proposition}

The bi-skew braces constructed in  this way comprise in particular the
bi-skew braces arising from the groups of~\cite{perfect}. Examples in
finite $p$-groups, like the following one, are easy to construct.
\begin{example}
  \label{ex:exexex}
  Let $p > 2$ be
  a prime, and let $K$ be an elementary abelian group of order $p^{n}$,
  for $p > n \ge 2$, with generators $a_{0}, \dots a_{n-1}$. There is a
  unique automorphism $\beta$ of $K$ such that $a_{i}
  \mapsto a_{i} a_{i+1}$ for $i = 0, \dots, n-1$ (where we take $a_{n} =
  1$), and $\beta$ has 
  order $p$. The group presented as
  \begin{equation*}
    G = \Span{ K, b : b^{p} = a_{n-1}, b^{-1} a b = a^{\beta} \text{ for
        $a \in K$}}
  \end{equation*}
  satisfies the hypotheses of Proposition~\ref{prop:central}, with $H =
  \Span{b}$, and does not split as a semidirect product over $K$, as $K$
  has index $p$ in $G$, and all
  elements outside $K$ have order $p^{2}$. It is easy to see that $(G,
  \circ)$ is abelian here.
\end{example}

\begin{proof}[Proof of Proposition~\ref{prop:central}]
  $\gamma$ is clearly well defined because the intersection of $H$ and
  $K$ lies in the centre.

  Now $\gamma$ restricted to $H$ is a RGF.
  Proposition~\ref{prop:lifting}
  shows that the extension~\eqref{eq:a_lifting} is a GF
  on $G$, which is readily seen to satisfy
  Theorem~\ref{eq:bi-skew-via-gamma}\eqref{item:anti}, as for $j_{1},
  h_{2} \in H$ and $k_{1}, k_{2} \in K$ we have
  \begin{align*}
    \gamma(h_{1} k_{1} h_{2} k_{2})
    &=
    \gamma(h_{1} h_{2} k_{1}^{h_{2}} k_{2})
    \\&=
    \iota(h_{1} h_{2})^{-1}
    \\&=
    \iota(h_{2})^{-1} \iota(h_{1})^{-1}
    \\&=
    \gamma( h_{2} k_{2}) \gamma(h_{1} k_{1}) .
  \end{align*}
  Therefore
  $\gamma$ is a bi-GF.

  If $\gammabar$ is a bi-GF, then $H = \ker(\gammabar) \norm
  G$. Conversely, if $H \norm G$, then we can exchange the roles of
  $H$ and $K$.
\end{proof}

\subsection{Pairs of compatible automorphisms of semidirect products}

With the second extension of Childs's result we are back to semidirect
products, where we allow further GF's on $H$.

We start by recalling a fact from~\cite[Theorem 1]{Curran} about
automorphisms of semidirect products. Let $G$ be a semidirect product of $K
\norm G$ by $H \le G$. An automorphism $d$ of $H$ can be
extended to an automorphism of $G$ that leaves $K$ invariant, acting
on $K$ as $a \in \Aut(K)$, if and only if the relation
\begin{equation}
  \label{eq:a-and-d}
  \iota(h)^{a} = \iota(h^{d}),
  \text{ for $h \in H$}
\end{equation}
holds, where
\begin{align*}
  \iota :\ &H \mapsto \Aut(K)\\
           &h \mapsto (k \mapsto k^{h}).
\end{align*}
Note that $d$ only determines $a$ up to an element of
$C_{\Aut(K)}(\iota(H))$.

We can  thus consider the  subset $\mathcal{P}$ of $\Aut(G)$  given by
all automorphisms of $G$ which leave $H$ and $K$ invariant, and act as
$d \in \Aut(H)$ on  $H$, and as $a \in \Aut(K)$ on  $K$, where $a$ and
$d$ are  related by~\eqref{eq:a-and-d}; we write  such an automorphism
as $d a$.

Now $\mathcal{P}$ is clearly a subgroup of
$\Aut(G)$, as if $d_{1} a_{1}, d_{2} a_{2} \in \mathcal{P}$, then for
their product  $d_{1} a_{1} d_{2} a_{2} = d_{1} d_{2} a_{1}
a_{2}$ \eqref{eq:a-and-d} yields, for $h \in H$,
\begin{equation*}
   \iota(h)^{a_{1} a_{2}} = \iota(h^{d_{1}})^{a_{2}} = \iota(h^{d_{1} d_{2}}),
\end{equation*}
so that $d_{1} d_{2} a_{1} a_{2} \in \mathcal{P}$.

We can now state
\begin{proposition}
  \label{prop:semi}
  Let $G$ be a semidirect product of $K \norm G$ by $H \le G$.

  Let $\gamma' : H \to \mathcal{P}$ be a RGF that satisfies
  \begin{equation*}
    \gamma'(h_{1} h_{2})
    =
    \gamma'(h_{2}) \gamma'(h_{1}).
    \text{ for $h_{1}, h_{2} \in H$}.
  \end{equation*}
  In particular, $\gamma'$  composed with the restriction to  $H$ is a
  bi-GF on $H$.

  Then the map $\gamma : G \to \Aut(G)$ defined by
  \begin{equation*}
    \gamma(h k) = \gamma'(h),
    \text{ for $h \in H$, $k \in K$}
  \end{equation*}
  is a bi-GF on $G$.
\end{proposition}
Proposition~\ref{prop:Childs} of Childs can be regarded as the special
case of this, where $\gamma'(h) = \iota(h^{-1})$ maps $H$ onto the subgroup
\begin{equation*}
  \Set{ x \mapsto h x h^{-1} : h \in H }
\end{equation*}
of $\mathcal{P}$.

In~\cite{CCDC} there are  examples  of the construction of this
Proposition,       related      to       the      ``lifting''       of
Proposition~\ref{prop:lifting}.

Example~\ref{ex:ex}  above is  of this
kind, where  $K$ is cyclic  of order $p^{2}$,  $H$ is cyclic  of order
$q$, where  $p$ and $q  \mid p  - 1$ are  primes, and $\gamma'(h)  : x
\mapsto h^{t} x h^{-t}$, for $h \in H$, for some $t \ne 1$. Here,
however, the component  on $H$ of the automorphism in $\mathcal{P}$ is trivial.

Another
instance occurs in~\cite[4.2.2]{CCDC}, which we describe in the
following
\begin{example}
  \label{ex:exex}
  Let $K = \Span{k}$  be a cyclic group of order $q$,  and $H = \Span{h}$
  be a cyclic group of order $p^{2}$,  where $p > 2$ and $q$ are primes,
  with $p \mid q  - 1$. Let $G$ be the semidirect  product of $K$ by
  $H$, where $H$ induces on $K$ a group of automorphisms of order $p$.

  Let $\psi \in \Aut(H)$ be the map $x \mapsto x^{1+p}$. Then for each
  $s, t$, it is immediate to see that $\iota(h)^{-s}_{\restriction K}
  \psi^{t}$ is  in $\mathcal{P}$, and that
  \begin{align*}
    \gamma' :\ &H \to \Aut(G)\\
               &h^{i} \mapsto \iota(h)^{-s i}_{\restriction K} \psi^{t i}
  \end{align*}
  is a RGF satisfying the hypothesis of the Proposition. It is shown
  in~\cite[4.2.2]{CCDC} that when $s \not\equiv 1 \pmod{p}$ we have
  $(G, \circ) \cong (G, \cdot)$, while for $s \equiv 1 \pmod{p}$, the
  group $(G, \circ)$ is cyclic.
\end{example}

\begin{proof}[Proof of Proposition~\ref{prop:semi}]
  Proposition~\ref{prop:lifting} yields that $\gamma$ is a GF on $G$.

  Now for $h_{1}, h_{2} \in H$ and $k_{1}, k_{2} \in K$ we have
  \begin{align*}
    \gamma( h_{1} k_{1} h_{2} k_{2})
    &=
    \gamma( h_{1} h_{2} k_{1}^{h_{2}} k_{2} )
    \\&=
    \gamma'( h_{1} h_{2} )
    \\&=
    \gamma'( h_{2} ) \gamma'(h_{1})
    \\&=
    \gamma( h_{2} k_{2} ) \gamma( h_{1} k_{1} ),
  \end{align*}
  so that $\gamma$ is a bi-GF on $G$.
\end{proof}

\subsection{Bi-skew braces from bi-homomorphisms into the centre}
\label{subs:bi-hom}

We now give  a construction that generalises the examples of
bi-skew braces coming from~\cite[Theorem~5.5]{p4}~and \cite{Tsang-p}.
\begin{prop}
  \label{thm:Delta}
  Let $G$ be a group,  and $K \le Z(G)$.

  Let 
  \begin{equation*}
    \Delta : G / K \times G / K  \to K
  \end{equation*}
  be a bi-homomorphism, that is, a function
  that is a homomorphism in each  of the two variables.

  Then the function
  $\gamma : G \to G^{G}$ given by
  \begin{equation}
    \label{eq:gamma-via-Delta}
    x^{\gamma(y)}
    =
    \Delta(x K, y K) x
  \end{equation}
  is    a     bi-GF,    which    satisfies    the     conditions    of
  Lemma~\ref{prop:gammaG-is-abelian}.
\end{prop}

It  is easy  to construct  examples of  this kind  in any  non-trivial
finite $p$-group $G$. Given any (non-trivial) subgroup $K \le Z(G)$,
any homomorphism
\begin{equation*}
  G / G' K \otimes G/ G' K \to K
\end{equation*}
will  yield a  $\Delta$ as  in the  Theorem; here  $G'$ is the derived
subgroup of $G$, and the tensor product is taken over $\Z$. In the
particular case when $K$ has exponent $p$, and $K \le \Frat(G)$,
where $\Frat(G)$ denotes the Frattini subgroup of $G$, we obtain
maps $\Delta$ from any linear map
\begin{equation*}
  G / \Frat(G) \otimes G /\Frat(G) \to K,
\end{equation*}
where $G / \Frat(G)$ and $K$ are regarded as vector spaces over the
field with $p$ elements

\begin{proof}[Proof of Proposition~\ref{thm:Delta}]
  For $x, y, z \in G$ we have, since $\Delta$ takes values in the centre
  of $G$,
  \begin{align*}
    (x y)^{\gamma(z)}
    &=
    \Delta( x y K, z K) x y
    \\&=
    \Delta( x K, z K) \Delta( y K, z K) x y
    \\&=
    \Delta( x K, z K) x \Delta( y K, z K) y
    \\&=
    x^{\gamma(z)} y^{\gamma(z)},
  \end{align*}
  that is, $\gamma(z) \in \End(G)$ for $z \in G$.

  Since $\Delta$ is a homomorphism in the second variable, we have for
  $x \in G$
  \begin{equation*}
    \Delta(x K, 1 K) \Delta(x K, 1 K)
    =
    \Delta(x K, 1 K),
  \end{equation*}
  so that $\Delta(x K, 1 K) = 1$. It follows that $\gamma(1)$ is the
  identity map on $G$.
  Note next that for $x, y, z \in G$ we have
  \begin{equation}
    \label{eq:xyz}
    \begin{aligned}
      z^{\gamma(y) \gamma(x)}
      &=
      \Delta(z^{\gamma(y)} K , x K) z^{\gamma(y)}
      \\&=
      \Delta(\Delta(z K, y K) z K, x K) \Delta(z K, y K) z
      \\&=
      \Delta(z K, x K) \Delta(z K, y K) z
      \\&=
      \Delta(z K, x y K) z
      \\&=
      z^{\gamma(x y)},
    \end{aligned}
  \end{equation}
  that   is,    $\gamma   :   (G,    \cdot)   \to   \End(G)$    is   a
  anti-homomorphism. Taking  $y = x^{-1}$  in~\eqref{eq:xyz}, we see
  that $\gamma(x) 
  \in \Aut(G)$ for $x \in G$.

  Since $\Delta$ takes values in the centre of $G$,
  we have
  \begin{equation*}
    \Delta(z K, x K) \Delta(z K, y K)
    =
    \Delta(z K, y K) \Delta(z K, x K),
  \end{equation*}
  in~\eqref{eq:xyz},   which  thus  also  yields  $\gamma(y)
  \gamma(x) = \gamma(x) \gamma(y)$, that is, $\gamma(G)$ is
  abelian. It follows from Lemma~\ref{prop:gammaG-is-abelian} that
  $\gamma$ is a bi-GF.
\end{proof}

\subsection{Completion of the proof of Corollary~\ref{cor:T}}
\label{subs:proof}

We deal first with Corollary~\ref{cor:T}\eqref{item-TG1-still}.
In~\cite[Proposition~7.9]{perfect}  a  group  $G   =  (G,  \cdot)$  is
constructed, which is  the central product of  two isomorphic, finite,
perfect groups $H, K$, where a  centre of order $3$ is amalgamated,
and it is shown that the  construction of
Proposition~\ref{prop:central} yields  a group  $(G, \circ)$  which is
not isomorphic to $G$. It follows that  $(G, \cdot, \circ)$ is a bi-skew brace
whose associated regular  subgroup $N \cong (G, \circ)$ is  not in the
image $\Hc(G)$ of $\mathcal{T}$.

For   this    group   $G$    ~\cite[Proposition~7.9]{perfect}   yields
$\Size{T(G)} = 2$. Now consider, with $G$ as above, the direct product
$S = G  \times L$, where $L$  is the direct product  of $n-1$ pairwise
not  isomorphic,  finite, perfect,  centreless  groups,  none of  them
isomorphic  to $H$.   Then  the results  of~\cite{perfect} yield  that
$\Size{T(S)}    =    2^{n}$.     However,    the    construction    of
Proposition~\ref{prop:central},   where  $S$   is   regarded  as   the
semidirect product with amalgamation of $K  L$ by $H$, shows that $S =
(S, \cdot)$  is the  additive group  of a  bi-skew brace,  $(S, \cdot,
\circ)$, whose associated regular subgroup $N \cong (S, \circ)$ is not
isomorphic  to $S$,  so  that $N$  is  not in  the  image $\Hc(S)$  of
$\mathcal{T}$.

For a different example, let $G$ be the free group of nilpotence class
$2$  and exponent  $3$ on  $n \ge  2$ generators.   Then~\cite[Theorem
  5.2]{p4} shows that $\Size{T(G)} = 2$.  However, the construction of
Proposition~\ref{thm:Delta}   (which  extends   to   this  situation   the
construction  of~\cite[Theorem~5.5]{p4})  shows   that  $G$  has  many
non-trivial bi-GF's.

For   Corollary~\ref{cor:T}\eqref{item-TG1-still},   an   example   is
provided  by  the  cyclic  group  $G$  of  order  $4$.   It  is  shown
in~\cite[Proposition  4.1(2)]{fgab}  that $T(G)  =  1$,  and that  the
holomorph of $G$,  which is isomorphic to the dihedral  group of order
$8$, contains a non-cyclic normal regular subgroup of order $4$.

\subsection{Nilpotent groups of class two as circle groups}

The construction of Proposition~\ref{thm:Delta} is  an analogue of the one
used in~\cite[Theorem 1]{AuWa} by J.~C.~Ault~and J.~F.~Watters to show
that certain groups of nilpotence class $2$ are the circle groups of a
(nilpotent) ring, and  in turn the circle groups of  a brace. (A brace
can be defined  as a skew brace with abelian  additive group, although
the concept of  a brace predates that of  a skew brace~\cite{braces}.)
It is  still an open  problem whether every  \emph{infinite} 
nilpotent
group  of  nilpotence  class  $2$  is the  circle  group  of  a  brace
(\cite[Problem 10.5]{Cedo-sur}, \cite[Problem 32]{Ven-sur}).

Proposition~\ref{thm:Delta} can be used to reformulate the results of Ault
and Watters in the context of bi-skew braces.  To give an example, let
$G = (G, \cdot)$ be a nilpotent group of nilpotence class $2$ which is
uniquely $2$-radicable, that  is, each element $g \in G$  has a unique
square root $g^{1/2}$. Let $K = Z(G)$. Then
\begin{align*}
  \Delta :\ &G/K \times G/K \to K\\
            &(x K, y K) \mapsto [x, y]^{-1/2}
\end{align*}
satisfies  the  hypotheses  of  Proposition~\ref{thm:Delta}.  (In practice
we are appealing to a
particular case of the  Baer correspondence~\cite{Baer-corr}, which is
in  turn  an  approximation  of  the  Lazard  correspondence  and  the
Baker-Campbell-Hausdorff formulas~\cite[Ch.~9~and 10]{khukhro}.)

For the
corresponding $\gamma$ and $\circ$ we have, for $x, y \in G$,
\begin{equation*}
  x \circ y
  =
  x^{\gamma(y)} y
  =
  \Delta(x K, y K) x y
  =
  [x, y]^{-1/2} x y,
\end{equation*}
and
\begin{equation*}
  y \circ x
  =
  y^{\gamma(y)} x
  =
  \Delta(y K, x K) y x
  =
  [y, x]^{-1/2} x y [y, x]
  =
  [x, y]^{-1/2} x y,
\end{equation*}
so that  $(G, \circ)$  is abelian.  (This would  also follow  from the
calculation of~\cite[Lemma 2.4]{p4}.)

Proposition~\ref{thm:Delta} yields that $(G, \cdot, \circ)$ is a bi-skew
brace, so that $(G, \circ, \cdot)$ is a brace with the original $(G,
\cdot)$ as the circle group.

\bibliographystyle{amsalpha}
 

\providecommand{\bysame}{\leavevmode\hbox to3em{\hrulefill}\thinspace}
\providecommand{\MR}{\relax\ifhmode\unskip\space\fi MR }
\providecommand{\MRhref}[2]{%
  \href{http://www.ams.org/mathscinet-getitem?mr=#1}{#2}
}
\providecommand{\href}[2]{#2}

\end{document}